\documentclass[11pt]{article}

\usepackage{pgf,tikz}
\usepackage{mathrsfs}
\usetikzlibrary{arrows}

\usepackage{todonotes, cite}

\usepackage{cite}
\usepackage{graphicx} 

\graphicspath{{figures/}} 

\usepackage{fullpage, url,amsmath,amsfonts,amssymb,mathtools,mathrsfs,graphicx,  algorithm, float, sansmath,epstopdf,color,caption,enumitem,tabularx}
\usepackage[final]{pdfpages}
\usepackage{amsthm}
\usetikzlibrary{automata,topaths}
\usetikzlibrary{decorations.pathreplacing,shapes.misc}
\usepackage{fancyhdr}

\def\Id{\mathbf{I}\mathbf{d}}
\def\proj{\mathbf{P}{\mathbf{r}}\mathbf{o}\mathbf{j}}
\def\beq{\begin{equation}}
\def\eeq{\end{equation}}
\def\baq{\begin{eqnarray}}
\def\eaq{\end{eqnarray}}
\def\baqn{\begin{eqnarray*}}
\def\eaqn{\end{eqnarray*}}

\newcommand{\ball}{\mathbb{B}}

\usepackage{multirow}
\usetikzlibrary{calc,arrows}
\theoremstyle{plain}
\newtheorem{definition}{Definition}
\newtheorem{remark}{Remark}

\newtheorem{example}{Example}
\newtheorem{theorem}{Theorem}
\newtheorem{lemma}[theorem]{Lemma}

\newtheorem{proposition}[theorem]{Proposition}

\usepackage{blindtext}

\usepackage[colorlinks,linkcolor=blue,citecolor=red]{hyperref}

\usepackage{xcolor, framed}

\newcommand{\R}{{\mathbb R}}
\newcommand{\N}{{\mathbb N}}

\newcommand{\interior}{{\rm int}\kern 0.06em}

\def\<{\langle}
\def\>{\rangle}

\newcommand {\graph} {{\rm Graph\,}\,}

\usepackage{ulem}

\usepackage{empheq}
\definecolor{myblue}{rgb}{.8, .8, 1}
\newcommand*\mybluebox[1]{
\colorbox{myblue}{\hspace{1em}#1\hspace{1em}}}

\usepackage{subcaption}

\usepackage{pict2e}

\if
{

\usepackage[pageref]{backref}
\renewcommand*{\backrefalt}[4]{%
\ifcase #1 %
(Not cited)%
\or
(Cited on p.~#2)%
\else
(Cited on pp.~#2)%
\fi
}

}
\fi

\begin{document}
\title{ {  Explicit Convergence  Rate of The  Proximal Point Algorithm  under  R-Continuity}}
\author{Ba Khiet Le \thanks{Optimization Research Group, Faculty of Mathematics and Statistics, Ton Duc Thang University, Ho Chi Minh City, Vietnam\vskip 0mm
 E-mail: \texttt{lebakhiet@tdtu.edu.vn}} \qquad Michel Th\' era \thanks{Mathematics and Computer Science Department, University of Limoges, 123 Avenue Albert Thomas,
87060 Limoges CEDEX, France \vskip 0mm
School of
Engineering, IT and Physical
Sciences, Federation University,
Ballarat, Victoria, 3350, Australia \vskip 0mm
 E-mail: \texttt{michel.thera@unilim.fr}}
 }
\date{}

\maketitle

\begin{abstract}
The paper provides a thorough comparison between R-continuity and other fundamental tools in optimization such as metric regularity, metric subregularity and calmness. We show that R-continuity has some advantages in the convergence rate analysis of algorithms solving optimization problems. We also present some properties of R-continuity and study the explicit convergence rate of the Proximal Point Algorithm $(\mathbf{P}\mathbf{P}\mathbf{A})$ under the R-continuity.
\end{abstract}

{\bf Keywords.} R-continuity, metric regularity, metric subregularity, calmness, Proximal Point Algorithm, convergence rate 

{\bf AMS Subject Classification.} 28B05, 34A36, 34A60, 49J52, 49J53, 93D20

\section{Introduction}\label{sec1}
In what follows, $\mathbb{X}$ and $\mathbb{Y}$  are  real Banach spaces whose norms are  designated by $\Vert\cdot\Vert$. We  use  the notation $\ball(x,r)$ to denote the closed ball with center $x$ and radius $r>0$  and by $\ball$  the  closed unit ball  which consists of the elements of norm less than or equal to $1$. By a set-valued mapping  $\mathcal{A}:\mathbb{X}\rightrightarrows\mathbb{Y}$, we mean a mapping which assigns to each $x\in \mathbb{X}$ a subset $\mathcal{A}(x) $ (possibly empty) of $\mathbb{Y}$. The domain, the range and  
the graph of $\mathcal{A}$ are defined respectively by 
$$
{\rm dom}\,\mathcal{A}=\{x\in \mathbb{X}: \mathcal{A}(x)\neq \emptyset \},\;\;{\rm rge}\,\mathcal{A}=\cup_{x\in \mathbb{X}}\mathcal{A}(x),
$$
and 
$${\rm Graph \,}{\mathcal{A}=\{(x,y)\in \mathbb{X}\times \mathbb{Y} \;\text{such that} \; y\in\mathcal{A}(x})\}.$$  
As usual we denote by $\mathcal{A}^{-1} :\mathbb{Y}\rightrightarrows \mathbb{X}$ the inverse of $\mathcal{A}$ defined by 
$$x\in \mathcal{A}^{-1}(y)\; \iff \;  y\in \mathcal{A}(x).$$ 
The notion $\mathbf{d}(x, \mathbf{S})$  stands for the distance from a point $x\in \mathbb{X}$  to a subset $\mathbf{S}\subset \mathbb{X}$:
$$\mathbf{d}(x, \mathbf{S}):=\, \inf_{y\in \mathbf{S}} \Vert x-y\Vert.$$
{Given   a set-valued mapping  $\mathcal{A}: \mathcal{H} \rightrightarrows \mathcal{H} $ defined in a Hilbert space $ \mathcal{H} $,  and inspired by Rockafellar's  paper \cite{Rockafellar}, recently B. K. Le \cite{L1}     introduced the notion of R-continuity   for studying the convergence rate of the Tikhonov regularization of  the inclusion}

\begin{empheq}[box =\mybluebox]{equation}\label{inclu}
0\in \mathcal{A}(x).
\end{empheq}
{ In \cite{L1},} it was proved that R-continuity is  a useful tool to analyze the convergence of $\mathbf{D}\mathbf{C}\mathbf{A}$  (Difference of Convex Algorithm) and   can {explain  why $\mathbf{D}\mathbf{C}\mathbf{A}$   is effective in approximating solutions for a broad class of functions.   In recent decades, there has been a surge of interest to study  variational inclusions such as (\ref{inclu}) since  they modelise
 a variety of important systems. This is the case especially in optimization when the condition for critical points is considered (the Fermat rule).  {Another connection with the inclusion (\ref{inclu}) arises in PDEs and is well discussed in the books
 \cite{Abm,br1}. The fact that  the \textit{solution set }$\mathbf{S}:={\mathcal{A}}^{-1}(0)$ of (\ref{inclu}) involves the inverse operator of $\mathcal{A}$, conducts  naturally to study the continuity of $\mathcal{{A}}^{-1}$ at zero. 
 {According to     \cite{L1},} the mapping  $ \mathcal{A}: \mathcal{H} \rightrightarrows \mathcal{H}$  is said to be \textit{R-continuous  at zero}  if there exist a radius $\sigma>0$ and a non-decreasing modulus function $\rho: \mathbb{R}^+\to \mathbb{R}^+$ satisfying $\lim_{r\to 0^+}\rho(r)=\rho(0)=0$  such that
  \begin{empheq}[box =\mybluebox]{equation}\label{rho}
\mathcal{A}(x) \subset \mathcal{A}(0)+\rho(\Vert x \Vert)\ball,\; \text{ for every}\; x\in \sigma \ball.
\end{empheq}
{Denoting by   $\mathbf{e}\mathbf{x}(C,D):=\sup_{x\in C}\mathbf{d}(x,D)$,   the excess of the set $C$ over the set $D$, with the convention that  $\mathbf{e}\mathbf{x}(\emptyset, D)=0 $  when $D$ is  nonempty  and $\mathbf{e}\mathbf{x}(C, \emptyset)=+\infty$ for any $C$, 
(\ref{rho}) is equivalent to saying that }
\begin{empheq}[box =\mybluebox]{equation}\label{rho1}
\mathbf{e}\mathbf{x}(\mathcal{A}(x), \mathcal{A}(0))\leq  \rho(\Vert x \Vert) \; \text{whenever}\; \Vert x\Vert \leq \sigma. 
\end{empheq}
  When     $\rho(r):=\,Lr$ for some $L>0$,  (\ref{rho1}) is changed into
 \begin{empheq}[box =\mybluebox]{equation}\label{R-Lipschitz}\mathbf{e}\mathbf{x}(\mathcal{A}(x), \mathcal{A}(0))\leq  L \Vert x \Vert \; \text{whenever}\; \Vert x\Vert \leq \sigma.
\end{empheq}
In this case, $\mathcal{A}$ is referred to as  \textit{R-Lipschitz continuous}  at zero or equivalently \textit{upper Lipschitz} at zero   in the sense of Robinson \cite{robinson} or  \textit{outer Lipschitz continuous} at zero \cite{adt}.
In addition, if  $\mathcal{A}(0)$ is a singleton,   R-Lipschitz continuity  at zero of  $\mathcal{A}$  is exactly the  Lipschitz continuity  {at zero} introduced by   R. T. Rockafellar in  \cite{Rockafellar} who considered Lipschitz continuity at zero of set-valued mappings as an important tool in optimization. Rockafellar demonstrated the linear convergence rate of the proximal point algorithm after a finite number of iterations from any starting point $x_0$.}
However the requirement that the solution set is a singleton is often quite restrictive in practice. Thus, allowing the solution set {$\mathbf{S}$ to be set-valued} and {and not requiring the continuity modulus
function $\rho$ to be Lipschitz continuous makes} R-continuity   a competitive alternative. It provides a viable option alongside other fundamental concepts such as metric regularity,  metric subregularity or calmness  in the study of sensitivity analysis as well as in establishing the convergence rate of algorithms. Note that R-continuity can be  also extended to  Banach spaces. In order to compare these regularity concepts, let us recall the definitions of metric regularity, metric subregularity and calmness (see, e.g., \cite{IOFFE, ioffe-outrata, Henrion, Zheng-Ng,  Zheng-Ng1, DmiKru08.1, sacha2016, RockBook2002, DonRoc09, NT,Penot2013,thibault} and the references therein).   Given $\mathcal{A}:\mathbb{X} \rightrightarrows \mathbb{Y}$  where $\mathbb{X}, \mathbb{Y}$ are real Banach spaces and $(\bar{x},\bar{y})\in {\rm Graph \,}\mathcal{A}$,  one says that $\mathcal{A}$  is  \textit{metrically regular}  near $(\bar{x},\bar{y})$ if  a linear error bound holds
 \begin{empheq}[box =\mybluebox]{equation}\label{metric}
\mathbf{d}(x,\mathcal{A}^{-1}(y))\le \kappa \mathbf{d}(y,\mathcal{A}(x)) 
\end{empheq}
for some $\kappa>0$ and for all $(x,y)$ close to $(\bar{x},\bar{y})$. {If (\ref{metric}) is satisfied for all $(x,y)\in \mathbb{X} \times \mathbb{Y}$, then $\mathcal{A}$  is  termed  \textit{globally metrically regular}.}
When  { $(\bar x, 0 )\in \graph \mathcal{A}$}, 
inequality (\ref{metric})  provides an estimate of how far is a point $x$ around $\bar{x}$ to the solution set $\mathcal{A}^{-1}(0)$. However, metric regularity can be too stringent as it requires the inequality (\ref{metric})  to hold in a neighborhood  of {$(\bar{x},\bar{y})$}. 
\vskip 2mm
One says that $\mathcal{A}: \mathbb{X} \rightrightarrows  \mathbb{Y}$ is \textit{calm}  at $(\bar{x},\bar{y})\in {\rm Graph \,} \mathcal{A}$  if  there  exist $\kappa>0, \epsilon >0, \sigma >0$ such that
 \begin{empheq}[box =\mybluebox]{equation}\label{calm}
 \mathcal{A}({x})\cap \ball(\bar y, \epsilon)\subset\mathcal{A}(\bar x)+ \kappa \Vert x-\bar x\Vert\ball,\; \; \;\text{for all}\; x\in \ball(\bar x,\sigma),
\end{empheq}
{or equivalently} 
 \begin{empheq}[box =\mybluebox]{equation}\label{golf1}
\mathbf{e}\mathbf{x}(\mathcal{A}(x) \cap \ball(\bar{y}, \epsilon),\mathcal{A}(\bar{x}))\le \kappa \Vert x-\bar{x}\Vert,  \;\;{\rm for \;all} \;x\in \ball(\bar{x}, \sigma).
\end{empheq}
Note that when the set  $\mathcal{A}(x) \cap \ball(\bar{y},\epsilon)$ is empty the  inequality (\ref{golf1})  is always satisfied. 
  When $\bar{x}=0$,  (\ref{golf1})  becomes
 \begin{empheq}[box =\mybluebox]{equation}\label{train}
\mathbf{e}\mathbf{x}(\mathcal{A}(x) \cap \ball(\bar{y}, \epsilon),\mathcal{A}(0))\le \kappa \Vert x \Vert,  \;\;{\rm  whenever } \;\Vert x\Vert \leq \sigma.
\end{empheq}
Another well-known regularity concept is metric subregularity. The  set-valued mapping $\mathcal{A}: \mathbb{X}\rightrightarrows \mathbb{Y}$  
 is  called \textit{metrically  subregular}  at $(\bar{x}, 0)\in \graph \mathcal{A}$ 
if there exists a constant $\kappa>0$ such that 
 \begin{empheq}[box =\mybluebox]{equation}\label{subregular}
\mathbf{d}(x,\mathcal{A}^{-1}(0))\le \kappa \mathbf{d}(0, \mathcal{A}(x)), \;\;\text{ for all }\;  x \;\text{close to}\; \bar{x}. 
\end{empheq}
 {It is known  (see, e.g., \cite[Proposition 2.62]{IOFFE})  that metric subregularity
  of  $ \mathcal{A} $ at $ (\bar x, 0)$ is equivalent  to 
 calmness of $ \mathcal{A}^{-1}$  at $(0, \bar x)$. 
{From (\ref{rho1}) and (\ref{train}), {let us observe} that  if $\mathcal{A}$ is R-Lipschitz continuous at 0 then it is calm at $(0,\bar{y})$ for any $\bar{y}\in \mathcal{A}(0)$.}
 While  metric regularity is too strong,  {calmness}  is relatively weak to deduce useful properties. For example when the set $\mathcal{A}(x) \cap \ball(\bar{y}, \epsilon)$ is empty for some $x\in \epsilon \ball$, no information can be {deduced}. In addition, with complicated $\mathcal{A}$, we do not know a specific $\bar{y}$ and have to use computers to find an approximation of an element of $\mathcal{A}(0)$.
The first advantage of  R-continuity is  {that it is} unnecessary to know a prior solution $\bar{y}$ in advance.  Furthermore in (\ref{rho}), since $x$ is small, for each $y\in \mathcal{A}(x)$, there exists {some}  $\tilde{y}\in \mathcal{A}(0)$  {close} to $y$. 
 {This fact is meaningful  since it is difficult to ensure $y$ {to be}  in the vicinity of $\bar y$, which is unknown.}
{Secondly, R-continuity  is straightforward and always guarantees  {the} system consistency in Hoffman's sense    \cite{Hoffman}.}
 {Indeed, Theorem \ref{consistent}  establishes that } under the R-continuity of $\mathcal{A}^{-1}$ at zero,  the inclusion (\ref{inclu}) is consistent, i.e., {if $y_\sigma$  has a small norm {and satisfies} $y_\sigma\in \mathcal{A}(x_\sigma)$,   then we can find a solution $\bar x\in \mathbf{S}$ such that  $x_\sigma$ is  {close}   to $\bar x$.  
  {Note that to guarantee the consistency property, {metric}  subregularity or metric regularity must be satisfied  globally.  Thirdly, R-continuity does not require the property (\ref{rho}) {to hold around} a point belonging {to} the graph of the operator; it is easily verified for a broad class of set-valued mappings.} 
In order  to achieve the metric regularity at some point $(\bar x, \bar y)\in {\rm Graph \,} \mathcal{A}$,  the celebrated  Robinson-Ursescu Theorem   \cite{robinson0,Ursescu} requires  for the operator $\mathcal{A}: \mathbb{X} \rightrightarrows  \mathbb{Y}$  {to} have {a} closed {and} convex graph  and  also that $\bar y \in {\rm int}(\mathcal{A}(\mathbb{X}))$ (the interior of $\mathcal{A}(\mathbb{X}))$.
 {R-continuity only necessitates that the graph of the operators be closed.} {Indeed, if  the operator has a closed graph at zero and is locally compact  at zero  then it is {R-continuous} at zero (Theorem \ref{compact})}. Conversely, if  the operator is {R-continuous} at zero and {its} value at zero is closed, then it has {a}  closed graph at zero (Theorem \ref{closed}).  {The condition for having its graph closed at zero is relatively mild.} If an operator's graph is closed, then it has a closed graph
 at zero as well. Furthermore, an operator has a closed graph if and only if its inverse has a closed graph. {Closed graph operators are usually found in optimization {when dealing with}  continuous single-valued mappings, maximally monotone operators (see, e.g., \cite{br}), the sum of two closed graph operators where one of them is single-valued, the sum of two closed graph set-valued operators where one of them is locally compact (Proposition \ref{closedgraph})}.  The Sign function in $\R^n$, which equals to the convex subdifferential of the norm function, is a well-known example of  {an operator}  with compact range and  {is} widely used in {image processing, mechanical and electrical engineering }(see, e.g., \cite{Acary,Micchelli}).

 {Finally, we demonstrate that R-continuity can be used to analyse the explicit convergence rate of the proximal point algorithm ($\mathbf{P}\mathbf{P}\mathbf{A}$)  when  applied to a maximally  monotone operator $\mathcal{A}:  \mathcal{H} \rightrightarrows \mathcal{H}$,  where $\mathcal{H}$ is a Hilbert space.  The $\mathbf{P}\mathbf{P}\mathbf{A}$, introduced by B. Martinet and further developed by Rockafellar, Bauschke and
 Combettes and others  (see, e.g., \cite{BC, Martinet1,Rockafellar}), is an essential tool in convex optimization if $\mathcal{A}$ is set-valued and lacks special structure. We show that if $\mathcal{A}^{-1}$} is globally R-Lipschitz continuous at zero (i.e., $\sigma=+\infty$),  the convergence of the proximal point algorithm is linear from the beginning (Theorem \ref{general}). {When considering the case where  $\mathcal{A}=\partial f$,  i.e. when  $\mathcal{A}$ is the convex subdifferential of  a proper lower semicontinuous {extended real-valued} convex function, not necessarily smooth  $f: \mathcal{H} \to \R \cup \{+\infty\}$,   if $\mathcal{A}^{-1}$ is only R-continuous, we obtain an explicit convergence rate of $\mathbf{P}\mathbf{P}\mathbf{A}$ based on the modulus function.}
If $(\partial f)^{-1}$ is R-Lipschitz {continuous}  at zero, one has the linear convergence of the generated sequence $(x_n)$ after some iterations (Theorem \ref{rate}).

The paper is organized as follows.  First we review the {necessary material about}  {R-continuity } of set-valued mappings and maximally monotone operators in Section \ref{s2}. Some key properties of R-continuity {are} established in Section \ref{s3}. The analysis of  {the} convergence rate of $\mathbf{P}\mathbf{P}\mathbf{A}$ under the  R-continuity is presented in Section \ref{s4}. The paper ends in Section \ref{s5} with some conclusions and perspectives. 

\section{Mathematical preliminaries}\label{s2}

{In what follows,  we extend the notion of R-continuity introduced in \cite{L1} from Hilbert spaces to Banach spaces, defined for any point in the domain of the operators. Let  $\mathcal{A}: \mathbb{X} \rightrightarrows \mathbb{Y}$ be a set-valued mapping  
where $\mathbb{X}, \mathbb{Y}$ are Banach spaces and ${\bar x}  \in {\rm dom}\,\mathcal{A}$.  

\begin{definition}
The set-valued mapping \noindent $\mathcal{A}:\mathbb{X} \rightrightarrows \mathbb{Y}$ is  called {R-continuous} at ${\bar x}  $ if there exist $\sigma>0$ and a non-decreasing function $\rho: \mathbb{R}^+\to \mathbb{R}^+$ satisfying $\lim_{r\to 0^+}\rho(r)=\rho(0)=0$  such that
 \begin{empheq}[box =\mybluebox]{equation}
\mathcal{A}(x) \subset \mathcal{A}({\bar x}  )+\rho(\Vert x-{\bar x}   \Vert)\ball, \;\forall x\in  \ball({\bar x}  ,\sigma),
\end{empheq}
or equivalently, for each $y\in \mathcal{A}(x)$, there exists ${\bar y}  \in  \mathcal{A}({\bar x}  )$ such that $\Vert y-{\bar y}  \Vert\le \rho(\Vert x-{\bar x}   \Vert)$ for all $x\in \ball({\bar x}  ,\sigma)$.

The  function $\rho$  is called a continuity modulus function of $\mathcal{A}$ at ${\bar x}  $ and $\sigma$ is called the radius. We say that $\mathcal{A}$ is {$R$-Lipschitz continuous} at ${\bar x}  $ with modulus $L$  if $\rho( r)=Lr$ for some $L>0$. In addition, if $\sigma=\infty$ then $\mathcal{A}$ is said globally {R-Lipschitz continuous} at ${\bar x}  $. 
\end{definition}

\begin{remark} 
The set-valued  mapping  $\mathcal{A}:\mathbb{X} \rightrightarrows \mathbb{Y}$ is   {R-continuous} at ${\bar x}  $ with modulus function $\rho$ and radius $\sigma$ iff $-\mathcal{A}$ is   {R-continuous} at ${\bar x}  $ with modulus function $\rho$ and radius $\sigma$. 
 \end{remark}
 \begin{definition}
The set-valued mapping  $\mathcal{A}:\mathbb{X} \rightrightarrows \mathbb{Y}$ is  called {R-continuous} if it is R-continuous at any point in its domain. It is called {R-Lipschitz continuous} if it is R-Lipschitz continuous at any point in its domain. In addition, if the radius $\sigma=+\infty$, then it is called  globally {R-Lipschitz continuous}.
\end{definition}
 \begin{proposition} \label{mer}
 If $\mathcal{A}: \mathbb{X} \rightrightarrows \mathbb{Y}$ is globally metrically regular then $\mathcal{A}^{-1}: \mathbb{Y} \rightrightarrows \mathbb{X}$ is globally R-Lipschitz continuous.
 \end{proposition}
 \begin{proof}
 Let be given ${y}, {\bar y}   \in {\rm dom}\, \mathcal{A}^{-1}$ and ${x}\in \mathcal{A}^{-1}({y})$. We have $ y\in \mathcal{A}x$ and since $\mathcal{A}$ is globally metrically regular, one has
 \beq
\mathbf{d}( x, \mathcal{A}^{-1}( \bar y)  )\le \kappa\mathbf{d}( \bar y,    \mathcal{A}( x)) \le \kappa \Vert y- \bar y  \Vert,
 \eeq
 for some $\kappa>0$. Since $ x\in \mathcal{A}^{-1} (y)$ is arbitrary, we deduce that 
 $$
 \mathbf{e}\mathbf{x}(\mathcal{A}^{-1}( y), \mathcal{A}^{-1}(\bar  y) )\leq \kappa \Vert y- \bar  y  \Vert
 $$
 and the conclusion follows. 
 \end{proof}
The following simple example shows  that the  metric regularity is strictly stronger than  R-Lipschitz continuity. This  important fact   means that we still obtain the convergence rate of optimization algorithms if only R-Lipschitz continuity or even R-continuity holds (see, e.g., \cite{L1}).
\begin{example}
Let $\mathcal{A}: \R \rightrightarrows \R$ be defined by  $$\mathcal{A}(x):=\,\left\{
\begin{array}{l}
[0,\infty)\;\;\;\;\;\;\;\;{\rm if}\;\; \;\;x=1\\ \\
0 \;\;\;\;\;\;\;\;\;\;\;\;\;\;\;\;{\rm if}\;\;\;-1<x<1\\ \\
(-\infty,0]\;\;\;\;\;\;{\rm if}\;\; \;\;x< 0.
\end{array}\right.$$  Then   $\mathcal{A}^{-1}(y)={\rm \mathbf{Sign}}(y)$ is  R-Lipschitz continuous for any positive modulus. However, if $y_n<0$ and $y_n\to 0$ then 
$$
\mathbf{d}( 1, \mathcal{A}^{-1}(y_n)  )=2, \;\;\mathbf{d}( y_n, \mathcal{A}(1)  )= \mathbf{d}( y_n, [0,\infty)  )=\vert y_n \vert \to 0.
$$ 
Thus $\mathcal{A}$ is not metrically regular at $(1,0)$.
\end{example}

 \noindent In the text that follows, we  {consider} set-valued mappings  with closed graph.
\begin{definition}
We say that $\mathcal{A}:\mathbb{X} \rightrightarrows \mathbb{Y}$ has a closed graph  if  $y_n\in \mathcal{A}(x_n),  \;y_n\to y$ and $x_n\to x$ then $y\in \mathcal{A}(x)$. It is said that $\mathcal{A}:\mathbb{X} \rightrightarrows \mathbb{Y}$ has a closed graph  at zero if  $y_n\in \mathcal{A}(x_n), \;x_n\to 0$ and $y_n\to y$   then $y\in \mathcal{A}(0)$.
\end{definition}
\begin{proposition}
The set-valued mapping  $\mathcal{A}: \mathbb{X} \rightrightarrows \mathbb{Y}$ has a closed graph  if and only if $\mathcal{A}^{-1}: \mathbb{Y} \rightrightarrows \mathbb{X}$ has a closed graph.
\end{proposition}
\begin{proof}
It follows directly from the definition of the inverse set-valued mapping $\mathcal{A}^{-1}$.
\end{proof}
\begin{proposition}\label{closedgraph}
Suppose that  $\mathcal{A}=\mathcal{A}_1+\mathcal{A}_2$ where $\mathcal{A}_1, \mathcal{A}_2: \mathbb{X} \rightrightarrows \mathbb{Y}$ has a closed graph;\\
a) If $\mathcal{A}_2$ is single-valued, then $\mathcal{A}$ has a closed graph;\\
b) If $\mathcal{A}_2$ is locally compact, i.e., for each $\bar{x}\in X$ there exists $\epsilon>0$ such that the set $\bigcup_{x\in \ball(\bar{x}, \epsilon)}\mathcal{A}_2(x)$ is compact, then $\mathcal{A}$ has a closed graph.
\end{proposition}
\begin{proof}
a) Let $y_n\in \mathcal{A}(x_n)=\mathcal{A}_1(x_n)+\mathcal{A}_2(x_n), y_n\to y$ and $x_n\to x$. We have $y_n=z_n+\mathcal{A}_2(x_n)$ for some $z_n\in \mathcal{A}_1(x_n)$. Since $z_n=y_n-\mathcal{A}_2(x_n)\to y-\mathcal{A}_2(x)$ and $\mathcal{A}_1$ has a closed graph, we imply that $ y-\mathcal{A}_2(x)\in \mathcal{A}_1(x)$ and the conclusion follows. \\
b) Similarly let $y_n\in \mathcal{A}(x_n)=\mathcal{A}_1(x_n)+\mathcal{A}_2(x_n), y_n\to y$ and $x_n\to x$. We have $y_n=z_n+v_n$ for some $z_n\in \mathcal{A}_1(x_n)$ and $v_n \in \mathcal{A}_2(x_n)$. Since $\mathcal{A}_2$ is locally compact, there exists a subsequence of $(v_n)$, without relabelling w.l.o.g,  converging to some $v\in \mathcal{A}_2(x)$. Thus $z_n=y_n-v_n$ converges to $y-v\in \mathcal{A}_1(x)$. Therefore $y \in \mathcal{A}_1(x)+v\in \mathcal{A}_1(x)+\mathcal{A}_2(x).$
\end{proof}
\noindent Finally some useful properties of monotone operators are reminded.  Let $\mathcal{H}$ be a Hilbert space,
\noindent {a set-valued mapping $\mathcal{A}: \mathcal{H}\rightrightarrows \mathcal{H}$  is called {\it monotone} provided
$$
\langle \bar x-\bar y,x-y \rangle \ge 0 \;\;\forall \;x, y\in \mathcal{H}, \bar x\in \mathcal{A}(x) \;{\rm and}\;\bar y\in \mathcal{A}(y).
$$
In addition, it is called {\it maximally monotone} if there is no monotone operator $\mathcal{B}$ such that the graph of $\mathcal{A}$ is strictly {included} in  the graph of $\mathcal{B}$.} The mapping $\mathcal{A}$ is called {\it $\gamma$-strongly monotone} if 
$$
\langle \bar x-\bar y,x-y \rangle \ge \gamma \Vert x-y\Vert^2 \;\;\forall \;x, y\in \mathcal{H}, \bar x\in \mathcal{A}(x) \;{\rm and}\;\bar y\in \mathcal{A}(y).
$$
{Note that such operators have been studied extensively because of their role in convex
analysis and certain partial differential equations.}
\vskip 2mm
\noindent {The {\it resolvent}  of a maximally monotone operator $\mathcal{A}$ is defined respectively by
$
\mathbf{J}_\mathcal{A}:=(\Id+\mathcal{A})^{-1}
$}. {It is  well-known  that resolvents are }  single-valued and non-expansive (see, e.g., \cite{Minty}).

A set-valued mapping $\mathcal{A}: \mathbb{X} \rightrightarrows \mathbb{Y}$  is called {\it $\gamma$-coercive}  with modulus $\gamma>0$  if
 for all $x, y \in {\rm dom}\,\mathcal{A}$ and for all $\bar x \in \mathcal{A}(x)$, $\bar y \in \mathcal{A}(y)$, we have 
\begin{equation}\label{golf}
\Vert \bar x-\bar y\Vert \ge \gamma \Vert x-y\Vert.
\end{equation}
In order to simplify the writing we will use the notation:  for all $x, y \in {\rm dom} \,\mathcal{A}$, we have 
$$
\Vert \mathcal{A}(x)-\mathcal{A}(y)\Vert \ge \gamma \Vert x-y\Vert,
$$
to describe this property.
It is easy to see that a  $\gamma$-strongly monotone operator is $\gamma$-coercive. {Another example is given by }matrices with full column rank (see, e.g., \cite{L1}).  \vskip 2mm Next {we extend  to  a pair of set-valued mappings the notion of monotonicity of a pair of single-valued operators introduced in  Hilbert spaces  by {Adly}-Cojocaru-Le \cite{acl}. }
{\begin{definition}
Let be given two set-valued {mappings} $\mathcal{B}, \mathcal{C}: \mathbb{X} \rightrightarrows \mathbb{Y}$. The pair $(\mathcal{B}, \mathcal{C})$  is called monotone if 
for all $x, y\in \mathbb{X}, x_b \in \mathcal{B}(x), y_b \in \mathcal{B}(y), x_c \in \mathcal{C}(x), y_c \in \mathcal{C}(y)$, we have 
\begin{equation}\label{pluie}
\Vert( x_b -y_b)+(x_c -y_c)\Vert^2\ge \Vert x_b -y_b\Vert^2+\Vert x_c -y_c\Vert^2.
\end{equation}
Equivalently, using the notation previously given (\ref{pluie}) is equivalent to 
$$
\Vert \mathcal{B}(x)-\mathcal{B}(y)+\mathcal{C}(x)-\mathcal{C}(y)\Vert^2\ge \Vert \mathcal{B}(x)-\mathcal{B}(y)\Vert^2+\Vert \mathcal{C}(x)-\mathcal{C}(y)\Vert^2.
$$
\end{definition}
Note that when  $\mathbb{Y}=\mathcal{H}$ is a Hilbert space, the monotonicity of  $(\mathcal{B}, \mathcal{C})$ is equivalent to  
$$
\langle \mathcal{B}(x)-\mathcal{B}(y), \mathcal{C}(x)-\mathcal{C}(y) \rangle \ge 0, \;\;\forall\;x, y \in \mathcal{H}.
$$
}
In addition, we say $(\mathcal{B}, \mathcal{C})$ is $\gamma$-strongly monotone $(\gamma>0)$ if 
$$
\langle \mathcal{B}(x)-\mathcal{B}(y), \mathcal{C}(x)-\mathcal{C}(y) \rangle \ge \gamma \Vert x-y \Vert^2, \;\;\forall\;x, y \in \mathcal{H}.
$$
{\begin{remark}
\begin{enumerate}
\item In Hilbert spaces, the monotonicity of  $(\mathcal{B}, \mathcal{C})$ means that the increments of  $\mathcal{B}$ and $\mathcal{C}$ does not form an obtuse angle.
\item If $(\mathcal{B}, \mathcal{C})$  is  monotone then for all $x, y\in  \mathbb{X}$, one  has
$$
\Vert \mathcal{B}(x)-\mathcal{B}(y)+\mathcal{C}(x)-\mathcal{C}(y)\Vert \ge \Vert \mathcal{B}(x)-\mathcal{B}(y)\Vert.
$$
\item   If $\mathcal{B}$ is monotone then the pair $(\mathcal{B}, \Id)$ is monotone.
\item The strong monotonicity of the  pair $(\mathcal{B}, \mathcal{C})$ {is  important for  the linear convergence of optimization algorithms  solving  the inclusion $0\in \mathcal{B}(x)$ \cite{acl}.
{When  $\mathcal{B}$ fails to be monotone,    with some suitable choice of $\mathcal{C}$}, the  pair $(\mathcal{B}, \mathcal{C}) $  {becomes} strongly monotone,  {as the following  example shows.}
}
 
\end{enumerate}
\end{remark}

 \begin{example}\label{ex2}
 Let $\mathcal{B}, \mathcal{C}: \R^2 \rightrightarrows \R^2$ be defined by
$$
\mathcal{B}(x_1,x_2) = \left( \begin{array}{ccc}
 \mathbf{Sign}(x_2)+3x_2+\sin \vert x_1\vert \\ \\
 \mathbf{Sign}(x_1)+3x_1 + \cos \vert x_2 \vert
\end{array} \right), \;\;\mathcal{C}(x_1,x_2) = \left( \begin{array}{ccc}
3x_2 \\ \\
3x_1
\end{array} \right) 
$$
where
$$
  \mathbf{Sign}(a)= \left\{
\begin{array}{l}
1\;\;\;\;\;\;\;\;{\rm if}\;\; \;\;a> 0\\ \\
${\rm [-1,1]}$ \;\;\;{\rm if}\;\;\;a=0,\\ \\
-1\;\;\;\;\;\;{\rm if}\;\; \;\;a< 0.
\end{array}\right.
$$
Then $\mathcal{B}, \mathcal{C}$ are not monotone but $(\mathcal{B}, \mathcal{C})$ is $6$-strongly monotone. Indeed for all $x=(x_1,x_2)$ and $y=(y_1,y_2)$, we have
\baqn
&&\langle \mathcal{B}(x)-\mathcal{B}(y), \mathcal{C}(x)-\mathcal{C}(y) \rangle \\
&\ge&9 ( x_2-y_2)^2+ 3(\sin \vert x_1\vert-\sin \vert y_1 \vert)(x_2-y_2)+9 ( x_1-y_1)^2+ 3(\cos \vert x_2\vert-\cos \vert y_2 \vert)(x_1-y_1)\\
&\ge& 6 \Vert x-y \Vert^2.
\eaqn
\end{example}

}
}

\section{Properties of R-Continuity} \label{s3}
First we show that the sum (and also the difference) of two R-continuous mappings is also R-continuous.
\begin{theorem}
Suppose that $\mathcal{A}_1, \mathcal{A}_2 :\mathbb{X} \rightrightarrows \mathbb{Y}$ are    {R-continuous} at $\bar x$ then $\mathcal{A}_1+\mathcal{A}_2$ is  also {R-continuous} at $\bar x$.
\end{theorem}
\begin{proof}
Let $\rho_1, \rho_2$ and $\sigma_1, \sigma_2$ be the modulus functions and radii of $\mathcal{A}_1$ and  $\mathcal{A}_2$ respectively. We set $\sigma:=\min \{ \sigma_1, \sigma_2\}$ and $\rho(r):=\max\{\rho_1(r),\rho_2(r)\}$ for all $r\ge 0$ then $\rho$ is non-decreasing and $\lim_{r\to 0^+}\rho(r)=\rho(0)=0$. Taking $x\in   \ball(\bar x,\sigma)$ and $y\in \mathcal{A}_1(x)+\mathcal{A}_2(x)$, then $y=y_1+y_2$ where $y_1\in \mathcal{A}_1(x)$ and  $y_2\in \mathcal{A}_2(x)$. Since $\mathcal{A}_1, \mathcal{A}_2$ are   {R-continuous} at $\bar x$ there exist $\bar y_1\in  \mathcal{A}_1(\bar x)$ and $\bar y_2\in  \mathcal{A}_2(\bar x)$ such that $\Vert y_1-\bar y_1\Vert\le \rho_1(\Vert x-\bar x \Vert)$ and $\Vert y_2-\bar y_2\Vert\le \rho_2(\Vert x-\bar x \Vert)$. Let $\bar y=\bar y_1+\bar y_2\in \mathcal{A}_1(\bar x)+\mathcal{A}_2(\bar x)$ then 
\beq
 \Vert y-\bar y\Vert\le \Vert y_1-\bar y_1\Vert+\Vert y_2-\bar y_2\Vert \le 2\rho(\Vert x-\bar x \Vert).
 \eeq
 It means that $\mathcal{A}_1+\mathcal{A}_2$ is   {R-continuous} at $\bar x$ with modulus function $2\rho$ and radius $\sigma$.
\end{proof}

{Next we show that R-continuity is satisfied for a large class of operators and has a closed connection  with the closed graph property  at zero.}

\begin{theorem}\label{compact}
If $\mathcal{A}:  \mathbb{X}\rightrightarrows \mathbb{Y}$ has  a closed graph  at zero and {is}  locally compact  at zero  then $\mathcal{A}$ is {$R$-continuous} at zero.
\end{theorem}
\begin{proof}
 {We define the function $\rho$ as follows: { set}  $\rho(0)=0$ and   if $\sigma>0$, {set }
 $$\rho(\sigma)=\inf\{ \delta>0: \mathcal{A}(x) \subset \mathcal{A}(0)+\delta\ball,\;\;\forall x\in \sigma \ball\}.$$
 It is easy to see that $\rho$ is well-defined and non-decreasing because $\mathcal{A}$ is locally compact at zero.  Since $\rho$ is non-decreasing and bounded  {from} below by $0$,   $\lim_{\sigma\to 0^+}\rho(\sigma)$ exists. {If we suppose } that $\mathcal{A}$  is not {R-continuous} at $0$,  {then} we must have $\lim_{\sigma\to 0^+}\rho(\sigma)=\delta^*>0$.} {Hence}   there exist two sequences $(x_n)$, $(y_n)$ such that $x_n \to 0$, $y_n\in \mathcal{A}(x_n)$  and 
\beq\label{closedr}
y_n\notin \mathcal{A}(0)+\frac{\delta^*}{2}\ball.
\eeq
Since $\mathcal{A}$ is locally compact at zero, { on relabeling if necessary, we may suppose that $(y_n)$ converges to  $\bar y \in   \mathcal{A}(0)$ since 
  $\mathcal{A}$ has a closed graph  at zero. This contradicts $(\ref{closedr})$ } and the  proof is complete. 
\end{proof}
\begin{theorem}\label{closed}
If {$\mathcal{A}: \mathbb{X}\rightrightarrows \mathbb{Y}$ is R-continuous} at zero and $\mathcal{A}(0)$ is closed then $\mathcal{A}$ has a closed graph  at zero.
\end{theorem}
\begin{proof}
Suppose that $y_n\to y$, $x_n\to 0$ and $y_n\in \mathcal{A}(x_n)$. From the R-continuity of $\mathcal{A}$,  we have 
$$
y_n\in \mathcal{A}(x_n)\subset \mathcal{A}(0)+\rho(\Vert x_n \Vert).
$$
Since  $y_n\to y$, $\rho(\Vert x_n \Vert)\to 0$  and  $\mathcal{A}(0)$ is closed, we deduce that  $y\in \mathcal{A}(0)$ and thus $\mathcal{A}$ has a closed graph  at zero.
\end{proof}
{Although R-continuity is  straightforward, it always guarantees  the consistency  of the system in Hoffman's sense    \cite{Hoffman}. }
\begin{theorem}\label{consistent}
Let $\mathcal{A}: \mathbb{X} \rightrightarrows \mathbb{Y}$ be a set-valued mapping. If $\mathcal{A}^{-1}$ is R-continuous at zero then the inclusion $0\in \mathcal{A}(x)$ is consistent in the sense of Hoffman, i.e., { if $y_\sigma$  has a small norm satisfying}   $y_\sigma\in \mathcal{A}(x_\sigma)$  then there exists a solution $\bar x\in \mathbf{S}:=\mathcal{A}^{-1}(0)$ such that  $x_\sigma$ is close to $\bar x$.
\end{theorem}
\begin{proof}
Suppose that $\mathcal{A}^{-1}$ is R-continuous at zero with  modulus function $\rho$ and  radius $\sigma$. Let  $ y_e\in \mathcal{A}(x_\sigma)$ and $\Vert y_e \Vert \le  \sigma $. We {claim}  that we can find some $\bar x $ such that $ 0\in \mathcal{A}(\bar x )$ and $\Vert x_\sigma-\bar x  \Vert$ is also small.   Indeed, we have 
$$
x_\sigma \in \mathcal{A}^{-1}( y_e)\subset  \mathcal{A}^{-1}(0)+\rho(\sigma)\ball=\mathbf{S}+\rho(\sigma)\ball
$$ 
which implies that 
$$
\mathbf{d}(x_\sigma,\mathbf{S})\le \rho(\sigma)
$$
and the conclusion follows. 
\end{proof}
\begin{remark}
Let us note that to obtain the consistency property, {metric}  subregularity or metric regularity have to be satisfied globally. Indeed suppose that {metric }subregularity (\ref{subregular}) holds and $\mathbf{d}(0, \mathcal{A}(x))$ is small but $x$ is not close to ${\bar x} $, we cannot conclude that $x$ is close to a solution.  
\end{remark}

In optimization, when we consider the inclusion $0\in \mathcal{A}(x)$, we want to know whether $\mathcal{A}^{-1}$ is R-continuous or even R-Lipschitz continuous. Here we provide some cases  where the global {R-Lipschitz continuity} is satisfied. It is known that the inverse of any square matrix is globally {R-Lipschitz continuous} \cite[Proposition 2.10]{L1}. Now we prove that the inverse of  any matrix is globally {R-Lipschitz continuous}. Note that this result cannot be deduced {from}  the  Robinson-Ursescu Theorem. Finally we give some nontrivial nonlinear examples (see also Proposition \ref{mer} and \cite[Theorem 7]{adt}).
\begin{theorem}\label{R-Lipschitz}
If $\mathbf{A}\in \R^{m\times n}$ is a matrix, then $\mathbf{A}^{-1}$ is globally {R-Lipschitz continuous}. 
\end{theorem}
\begin{proof}
Let be given $\bar y\in   {\rm dom}\,\mathbf{A}^{-1}$. For all $y\in {\rm dom\,}(\mathbf{A}^{-1})$ and $x\in \mathbf{A}^{-1}(y)$, we want to find $\bar x\in \mathbf{A}^{-1}(\bar y)$ such that 
$$
\Vert x-\bar x\Vert \le \kappa \Vert y-\bar y\Vert, 
$$
for some $\kappa>0$. First we take some $x'\in \mathbf{A}^{-1}(\bar y)$. We have $y=\mathbf{A}x$ and $\bar y=\mathbf{A}x'$. Note that $\mathbf{A}^T\mathbf{A}\in \R^{n\times n}$ is a positive semi-definite matrix.  Let $x=x_i+x_k$ where $x_i,  x_k$ are the projections of $x$ onto ${\rm Im}(\mathbf{A}^T\mathbf{A})$ and ${\rm Ker}(\mathbf{A}^T\mathbf{A})={\rm Ker}(\mathbf{A})$, respectively. Similarly we decompose  $x'=x'_i+x'_k$. Then we have 
$$
\Vert\mathbf{A}^T\Vert \Vert y-\bar y\Vert \ge  \Vert\mathbf{A}^T(y-\bar y)\Vert =\Vert \mathbf{A}^T\mathbf{A}(x_i-x_i')\Vert \ge k \Vert x_i-x_i'\Vert
$$
for some $k>0$ (see, e.g., \cite[Lemma 2.8]{L1} or \cite[Lemma 3]{tbp}).
Thus if we choose $\bar x=x_i'+x_k$ then $\mathbf{A}\bar x=\mathbf{A}(x_i'+x_k)=\mathbf{A}(x_i'+x'_k)=\mathbf{A}x'=\bar y$ since $x_k, x'_k \in {\rm Ker}(\mathbf{A})$. In addition, we have
$$
\Vert x-\bar x\Vert =\Vert x_i-x_i'\Vert\le \frac{\Vert \mathbf{A}^T\Vert}{k}\Vert y-\bar y\Vert,
$$
and the conclusion follows. 
\end{proof}

\begin{proposition}
If $\mathcal{A}: \R^n \rightrightarrows \R^m$ has the form $\mathcal{A}=\mathbf{B}+\mathcal{F}\circ \mathbf{B}$ where $\mathbf{B}\in \R^{m\times n}$ is a matrix and $\mathcal{F}: \R^m \rightrightarrows \R^m$ is monotone then $\mathcal{A}^{-1}$ is globally R-Lipschitz continuous. 
\end{proposition}
\begin{proof}
Suppose that $y\in \mathcal{A(}x)=\mathbf{B}x+y_1, \;y_1\in \mathcal{F}(\mathbf{B}x)$ and $\bar y\in {\rm dom\,}(\mathcal{A}^{-1})$. We want to find $\bar x\in \mathcal{A}^{-1}(\bar y)$ such that 
$$
\Vert x-\bar x\Vert \le \kappa \Vert y-\bar y\Vert, 
$$
for some $\kappa>0$. First we take some $x'\in \mathcal{A}^{-1}(\bar y)$, i.e., $\bar y\in \mathcal{A}(x')=\mathbf{B}x'+\bar y_1, \;\bar y_1\in \mathcal{F}(\mathbf{B}x')$. Using the proof of Theorem \ref{R-Lipschitz}, there exists $\bar x$ such that $\mathbf{B}\bar x=\mathbf{B}x'$ and $\Vert \mathbf{B}x-\mathbf{B}\bar x\Vert \ge \kappa_1 \Vert x-\bar x \Vert$ for some $\kappa_1>0$.  Since $\mathcal{F}$ is monotone, one has 
$$
\Vert y-\bar y\Vert^2=\Vert \mathbf{B}x-\mathbf{B}\bar x+ y_1-\bar y_1\Vert^2 \ge  \Vert \mathbf{B}x-\mathbf{B}\bar x\Vert^2\ge \kappa_1^2\Vert x-\bar x \Vert^2
$$
and  the conclusion follows. 
\end{proof}
\begin{proposition}
If $\mathcal{A}:  \mathbb{X}\rightrightarrows \mathbb{Y}$ has the form $\mathcal{A}=\mathcal{B}+\mathcal{C}$ where  $\mathcal{B}: \mathbb{X} \rightrightarrows \mathbb{Y}$ is coercive and the pair $(\mathcal{B},\mathcal{C})$ is monotone then $\mathcal{A}^{-1}$ is globally R-Lipschitz continuous.   
\end{proposition}
\begin{proof}
Suppose that $\bar y=\bar y_1+\bar y_2\in \mathcal{A}(\bar x)$ where $\bar y_1\in \mathcal{B}(\bar x), \bar y_2\in \mathcal{C}(\bar x)$ and $y=y_1+y_2\in \mathcal{A}(x)$ where $y_1\in \mathcal{B}(x), y_2\in \mathcal{C}(x)$. Since $\mathcal{B}: \mathbb{X} \to \mathbb{Y}$ is coercive and the pair $(\mathcal{B},\mathcal{C})$ is monotone, one has
\baqn
\Vert \bar y-y\Vert^2&=&\Vert \bar y_1-y_1+\bar y_2-y_2\Vert^2\\
&\ge& \Vert \bar y_1-y_1\Vert^2\\
&\ge& \kappa^2 \Vert \bar x-x\Vert^2,
\eaqn
for some $\kappa>0$. Thus  $\mathcal{A}^{-1}$ is single-valued Lipschitz continuous and thus is R-Lipschitz continuous. 
\end{proof}

\section{Explicit Convergence Rate of $\mathbf{P}\mathbf{P}\mathbf{A}$ under the  R-Continuity} \label{s4}
Let $\mathcal{A}: \mathcal{H} \rightrightarrows \mathcal{H}$ be a maximally monotone operator with nonempty $\mathbf{S}:=\mathcal{A}^{-1}(0)$ where $\mathcal{H}$ is a Hilbert space. 
{Let us recall that  $\mathbf{P}\mathbf{P}\mathbf{A}$  generates for any starting point  $x_0$,  a sequence $(x_n)$ defined by the rule:
\beq
x_{n+1}=\textbf{J}_{\gamma \mathcal{A}}(x_n), \;x_0 \in \mathcal{H},
\eeq
for some $\gamma>0$ where $\textbf{J}_{\gamma \mathcal{A}}(x):\,=(\Id+\gamma \mathcal{A})^{-1}(x)$.}
$\mathbf{P}\mathbf{P}\mathbf{A}$ plays an important role  in convex optimization if $\mathcal{A}$ is set-valued and has no special structure. The following result shows that the convergence rate of $\mathbf{P}\mathbf{P}\mathbf{A}$ is indeed linear if  $\mathcal{A}^{-1}$ is R-Lipschitz continuous at zero globally.
\begin{theorem}\label{general}
 Suppose that  $\mathcal{A}^{-1}$ is R-Lipschitz continuous at zero globally with modulus function $\rho(r)=Lr$ for some $L>0$. {Then} for $n\ge 1$, the {sequence $(\mathbf{d}(x_n,\mathbf{S}))_{n\in \N}$ }converges to zero with linear rate where $(x_n)$ is the sequence generated by   $\mathbf{P}\mathbf{P}\mathbf{A}$ with $\gamma>2L$. 
\end{theorem}
\begin{proof}
Let $\bar x_n=\proj _\mathbf{S}(x_n)$. We have  $ \Vert x_{n+1}-\bar x_n\Vert \le \Vert x_{n}-\bar x_n\Vert = \mathbf{d}(x_{n},\mathbf{S})$ where the  inequality is obtained by using the nonexpansiveness of the resolvent.   We have 
\baq
-\frac{x_{n+1}-x_n}{\gamma}\in   \mathcal{A}(x_{n+1}).
\eaq
Since $\mathcal{A}^{-1}$ is R-Lipschitz continuous at zero globally, we obtain
$$
x_{n+1}=\mathcal{A}^{-1}\Big(-\frac{x_{n+1}-x_n}{\gamma}\Big)\subset  \mathcal{A}^{-1}(0)+\rho\Big(\frac{\Vert x_{n+1}-x_n\Vert}{\gamma}\Big)\ball=\mathbf{S}+\rho\Big(\frac{\Vert x_{n+1}-x_n\Vert}{\gamma}\Big)\ball.
$$
Consequently 
$$
\mathbf{d}(x_{n+1},\mathbf{S}) \le \frac{L\Vert x_{n+1}-x_n\Vert}{\gamma}\le  \frac{L}{\gamma}\Big(\Vert x_{n+1}-\bar x_n\Vert+\Vert x_{n}-\bar x_n\Vert\Big)\le  \frac{2L}{\gamma}\Vert x_{n}-\bar x_n\Vert=\kappa  \mathbf{d}(x_{n},\mathbf{S})
$$
where $\kappa=\frac{2L}{\gamma}<1$ by choosing $\gamma>2L$ .
\end{proof}
Next we consider the convergence rate of $\mathbf{P}\mathbf{P}\mathbf{A}$ under only the R-continuity of $\mathcal{A}^{-1}$. The following result is similar to \cite[Lemma 3.1]{aal}. Here we give a proof. 
\begin{lemma}\label{fzero}
Let $f: [t_0, \infty)\to \R^+$ be a nonincreasing function which is locally absolutely continuous and belongs to $\mathbf{L}^1(t_0, \infty)$. Then $\lim_{t\to \infty} tf(t)=0.$
\end{lemma}
\begin{proof}
Without loss of generality, suppose that $t_0\ge 0$ since the limit involves only for $t$ large. Let $F(t)=tf(t)$ then $F$ is locally absolutely continuous, bounded from below  by $0$ and for almost all $t\ge t_0$, we have 
$$
\frac{d}{dt}(F(t))=f(t)+tf'(t)\le f(t).
$$
Since $f\in \mathbf{L}^1(t_0, \infty)$, using \cite[Lemma 5.1]{aas}, we deduce that $\lim_{t\to\infty}F(t)$ exists. Suppose that $\lim_{t\to\infty}F(t)=c>0$ then there exists $T>0, c_1>0$ such that for all $t\ge T$ one has $F(t)\ge c_1$, or equivalently $f(t)\ge c_1 t^{-1}$. Then 
$$
\int_T^\infty f(t)dt\ge c_1\int_T^\infty  t^{-1}dt=\infty,
$$
a contradiction. Thus we must have  $\lim_{t\to\infty}F(t)=0$ or equivalently $\lim_{t\to\infty}tf(t)=0$.
\end{proof}
\begin{lemma}\label{ratea}
Let $(a_n)$ be a nonincreasing  nonnegative sequence such that $\sum_{n=1}^\infty a_n < \infty$. Then we have $\lim_{n\to \infty}na_n=0$.
\end{lemma}
\begin{proof}
Since $\sum_{n=1}^\infty a_n < \infty$, we have $\sum_{n=1}^\infty a_{n+1} < \infty$ and thus $\sum_{n=1}^\infty b_n < \infty$ where $b_n:=(a_n+a_{n+1})/2$. Let $f:[1, \infty)\to \R^+$ be defined as follows: if $t\in [n, n+1]$ then $f(t)=a_n+(a_{n+1}-a_n)(t-n)$, $n\ge 1$. Then $f$ is a nonincreasing function which is locally absolutely continuous and belongs to $ \mathbf{L}^1(0, \infty)$ since 
$$
\int_1^\infty f(t)dt=\sum_{n=1}^\infty b_n < \infty.
$$
Note that if $t\in [n-1, n)$ then $f(t)\ge a_n$ and $tf(t)\ge (n-1)a_n$. Using Lemma \ref{fzero}, we have $\lim_{t\to\infty}tf(t)=0$ and thus
$$\lim_{n\to \infty}(n-1)a_n=0.$$
Therefore $\lim_{n\to \infty}na_n=0$ since  $\lim_{n\to \infty}a_n=0$.
\end{proof}
\begin{remark}
The  {fact that the sequence $(a_n)$  is nonincreasing}  cannot be {omitted}. For example let $n_k=k^2, k=1, 2, 3\ldots$ and $x_n=0$ if $n\neq n_k$  and $x_n=\frac{1}{n}$ if $n=n_k$. Then 
$$
\sum_{n=1}^\infty x_n =\sum_{k=1}^\infty  \frac{1}{k^2}<\infty.
$$
However $nx_n=0$ if $n\neq n_k$  and $nx_n={1}$ if $n=n_k$. Thus the sequence $(nx_n)$ is not convergent. 
\end{remark}
\begin{theorem}\label{rate}
Suppose that $f: \mathcal{H} \to \R \cup \{+\infty\}$ is a proper lower semicontinuous convex function {such that } ${\rm inf}f>-\infty$ and $(\partial f)^{-1}$ is R-continuous at zero with modulus function $\rho$ and radius $\sigma$. Let $(x_n)$ be the sequence generated by the proximal point algorithm
\beq
x_0\in \mathcal{H},\; x_{n+1}={\rm \textbf{J}}_{\gamma\partial f}(x_n), \;{\rm for\;some\; }\gamma >0.
\eeq
  We have \\
a) $\Vert x_{n+1} -x_{n}  \Vert^2=o(\frac{1}{n}).$\\
b)  Let $n_0$ be such that $\frac{\Vert x_{n_0+1} -x_{n_0}  \Vert}{\gamma}\le \sigma$. Then for all $n\ge n_0$, we have 
$$
\mathbf{d}(x_n,\mathbf{S})\le\rho \Big(o\Big(\frac{1}{\sqrt{n}}\Big)\Big)\to 0
$$
and 
$$
f(x_{n}) -f^*\le \mathbf{d}(x_n,\mathbf{S}) o(\frac{1}{\sqrt{n}})\le\rho\Big(o\Big(\frac{1}{\sqrt{n}}\Big)\Big) o\Big(\frac{1}{\sqrt{n}}\Big)\to 0
$$
where $f^*=\min_{x\in \mathcal{H}} f(x)$.

\noindent c) If $\rho(r)=Lr$ for some $L>0$ and $\gamma>2L$  then for $n\ge n_0$, the distance $\mathbf{d}(x_n,\mathbf{S})$ converges to zero with linear rate and so is $f(x_{n}) -f^*$. 
\end{theorem}
\begin{proof}
a) From 
\baq\label{ppa}
-\frac{x_{n+1}-x_n}{\gamma}\in  \partial f(x_{n+1})
\eaq
and the fact that $f$ is convex, we imply that 
$$
\left\langle -\frac{x_{n+1}-x_n}{\gamma}, x_n -x_{n+1} \right\rangle \le f(x_n)-f(x_{n+1}).
$$
Hence 
$$
\sum_{n=1}^\infty \Vert x_{n+1} -x_{n}  \Vert^2\le \gamma(f(x_0)-{\rm inf}f)<\infty
$$
and thus $\Vert x_{n+1} -x_{n}  \Vert$ converges to zero. Since $(\Vert x_{n+1} -x_{n}  \Vert)_n$ is nonincreasing due to the nonexpansiveness of the resolvent, using Lemma \ref{ratea}, we have
$$
\Vert x_{n+1} -x_{n}  \Vert^2=o(\frac{1}{n}).
$$
b) From (\ref{ppa}) and the  R-continuity of $(\partial f)^{-1}$, for $n\ge n_0$ , we obtain 
$$
x_{n+1}=(\partial f)^{-1} \left(-\frac{x_{n+1}-x_n}{\gamma} \right)\subset (\partial f)^{-1}(0)+\rho \left(\frac{\Vert x_{n+1}-x_n\Vert}{\gamma} \right)\ball=\mathbf{S}+\rho\left(\frac{\Vert x_{n+1}-x_n\Vert}{\gamma} \right)\ball.
$$
Therefore 
\beq\label{estid}
\mathbf{d}(x_{n+1},\mathbf{S}) \le \rho\Big(\frac{\Vert x_{n+1}-x_n\Vert}{\gamma}\Big)=\rho\Big(o(\frac{1}{\sqrt{n}})\Big)\to 0,\;\;{\rm as}\;\; n \to \infty. 
\eeq
Let $x_{n+1}^*=\proj _\mathbf{S}(x_{n+1})$ then $f(x_{n+1}^*)=f^*$. Using  (\ref{ppa}), one has 
$$
\left \langle -\frac{x_{n+1}-x_n}{\gamma}, x_{n+1}^* -x_{n+1}\right  \rangle \le f(\bar x)-f(x_{n+1}).
$$
which implies that 
$$ f(x_{n+1}) -f^* \le \left\Vert \frac{x_{n+1}-x_n}{\gamma} \right\Vert \mathbf{d}(x_{n+1},\mathbf{S}) \le\rho\Big(o\Big(\frac{1}{\sqrt{n}}\Big)\Big) o\Big(\frac{1}{\sqrt{n}}\Big)\to 0.$$
c)  Let $\bar x_n=\proj _\mathbf{S}(x_n)$. We know that  $ \Vert x_{n+1}-\bar x_n\Vert \le \Vert x_{n}-\bar x_n\Vert = \mathbf{d}(x_{n},\mathbf{S})$ .  From (\ref{estid}), for $n\ge n_0$, we have  
$$
\mathbf{d}(x_{n+1},\mathbf{S}) \le \frac  {L\ {}\Vert x_{n+1}-x_n\Vert}{\gamma} \le \frac{L}{\gamma} \Big(\Vert x_{n+1}-\bar x_n\Vert+\Vert x_{n}-\bar x_n\Vert \Big)\le  \frac{2L}{\gamma}\Vert x_{n}-\bar x_n\Vert=\kappa \mathbf{d}(x_{n},\mathbf{S})
$$
where $\kappa=\frac{2L}{\gamma}<1$ by choosing $\gamma>2L$ . \\
\end{proof}
\begin{remark}
R-continuity  {ensures}  $\mathbf{P}\mathbf{P}\mathbf{A}$ (and also $\mathbf{D}\mathbf{C}\mathbf{A}$ \cite{L1}) {to be} consistent in the numerical sense, i.e., if $\Vert x_{n+1}-x_{n}\Vert$ is small then $x_n$ is close to a solution. It provides an estimation of the distance between $x_n$ and the solution set based on the convergence rate of $\Vert x_{n+1}-x_{n}\Vert$ to zero. 
\end{remark}

 \section{Conclusions} \label{s5}
\noindent In this paper, we highlighted several  advantages of R-continuity compared to other key  tools used  in optimization,  such as metric regularity, metric subregularity and calmness. We explored important properties of R-continuity  and derived   an explicit convergence rate for the Proximal Point Algorithm ($\mathbf{P}\mathbf{P}\mathbf{A}$) {under this  framework. We believe that the technique developed in the paper can be extended to gain further insights into   the convergence rate of other   optimization algorithms.

\section{Acknowledgements}
This research benefited from the support of the FMJH Program Gaspard Monge for optimization and operations research and their interactions with data science.
\bibliography{proximal.bib}
\bibliographystyle{plain}

\end{document}